\documentclass[12pt,a4paper]{article}
\usepackage{amssymb, graphicx}
\usepackage{latexsym}
\usepackage{amsmath}
\usepackage{amssymb}   % \mathbb
\usepackage{amsthm}
\usepackage{amscd}

\numberwithin{equation}{section}
\newtheorem{theorem}{Theorem}[section]

\newtheorem{definition}[theorem]{Definition}

\newtheorem{lemma}[theorem]{Lemma}

\begin{document}

\title{Some modular inequalities in Lebesgue spaces with variable exponent}

\author{Mitsuo Izuki, Takahiro Noi and Yoshihiro Sawano}

\maketitle

\begin{abstract}
Our aim is to study the modular inequalities
for some operators, for example 
the Bergman projection acting on, 
in Lebesgue spaces with variable exponent. 
Under proper assumptions on the variable exponent, 
we prove that the modular inequalities 
are hold if and only if the exponent almost everywhere equals to a constant.
In order to get the main results, we prove a lemma for a lower pointwise 
bound for these operators of a characteristic function.
\end{abstract}

%%%%%%%%%%%%%%%%%%%%%%%%%%%%%%%%%%%%%%%%%%%%%%%%%%%%%

\section{Introduction}

The study on variable exponent analysis
has been rapidly developed 
after the work 
\cite{KR} 
where Kov$\acute{\rm{a}}\check{\rm{c}}$ik and R$\acute{\rm{a}}$kosn\'ik
have established fundamental properties 
of variable Lebesgue spaces
(see also \cite{CF-book,INS2014,Sawano-text-2018}). 
In particular the theory of variable function spaces
in connection with
the boundedness of the Hardy--Littlewood 
maximal operator $M$
has been deeply studied. 
Cruz-Uribe, Fiorenza and Neugebauer \cite{CFN2003,CFN2004}
and Diening \cite{DieningMIA2004}
have independently obtained 
the log-H\"{o}lder continuous conditions 
that guarantee 
the boundedness of $M$ on variable Lebesgue spaces.
We also note that
the recent development 
of variable exponent analysis 
has 
the extrapolation theorem
from weighted inequalities to 
norm inequalities on variable Lebesgue spaces
(\cite{CFMP,CMP-book}).

In general, the boundedness of $M$ on the
variable Lebesgue space $L^{p(\cdot)}(\mathbb{R}^n)$
describes that the norm inequality 
\begin{equation}
\| Mf \|_{L^{p(\cdot)}(\mathbb{R}^n)}
\le C\, \| f \|_{L^{p(\cdot)}(\mathbb{R}^n)}
\label{intro-norm}
\end{equation}
holds for all $f\in L^{p(\cdot)}(\mathbb{R}^n)$, 
where $C$ is a positive constant independent of $f$.
Lerner \cite{Lerner}
has pointed out the crucial difference 
between the norm inequality (\ref{intro-norm})
and the following modular inequality
\begin{equation}
\int_{\mathbb{R}^n} Mf(x)^{p(x)}\, dx
\le
C\int_{\mathbb{R}^n} |f(x)|^{p(x)}dx.
\label{intro-modular}
\end{equation}
More precisely,
Lerner has proved that 
$p(\cdot)$ must be a constant function
whenever 
$\displaystyle 1 < \mathop{\mathrm{ess \,inf}}_{x\in \mathbb{R}^n}p(x)
\le \mathop{\mathrm{ess \, sup}}_{x\in \mathbb{R}^n}p(x)<\infty$ and
the modular inequality (\ref{intro-modular})
holds.
Izuki \cite{IzukiGMJ} has considered the 
difference for some operators arising from the wavelet theory.
Izuki, Nakai and Sawano \cite{INS2013,INS2014}
have given an alternative proof of Lerner's result.
They have also studied the problem in the weighted case (\cite{INS2015}).

Recently, Izuki, Koyama, Noi and Sawano \cite{IKNS2019} have considered some modular inequalities for some operators. 
In this paper, we focus on three operators below. 
First, we investigate the Bergman projection operator on the unit disc ${\mathbb D}$ in the complex plane. 
The generalization 
of holomorphic function spaces 
in terms of variable exponent 
and the boundedness of Bergman projection operators
on variable exponent spaces 
have been studied (\cite{ChaconRafeiro2014,CR2016,CRV2017,KS2016,KS2017}).
Among them we focus on the work \cite{ChaconRafeiro2014}
due to Chac\'on and Rafeiro.
They defined 
Bergman spaces 
$A^{p(\cdot)}(\mathbb{D})$
with variable exponent $p(\cdot)$
on the open unit disk $\mathbb{D}$.
Applying the local log-H\"older continuous condition and 
the extrapolation theorem, 
they proved 
the density of the set of polynomials in $A^{p(\cdot)}(\mathbb{D})$
and the boundedness of the 
Bergman projection $P\, : \, L^{p(\cdot)}(\mathbb{D})\to A^{p(\cdot)}(\mathbb{D})$
and the $\alpha$-Berezin transform 
$B_{\alpha}\, : \, L^{p(\cdot)}(\mathbb{D})\to L^{p(\cdot)}(\mathbb{D})$.
In particular, 
Chac\'on and Rafeiro 
\cite{ChaconRafeiro2014} have obtained the norm inequality
\begin{equation}
\| Pf \|_{L^{p(\cdot)}(\mathbb{D})}
\le C\, \| f \|_{L^{p(\cdot)}(\mathbb{D})}
\label{intro-norm-analytic}
\end{equation}
for all $f\in L^{p(\cdot)}(\mathbb{D})$.

Second our target operator is 
\[
B_{{\mathbb R}^2_+}f(z) = \frac{-1}{\pi} \int_{{\mathbb R}_+^2}\frac{f(w)}{(z-\overline{w})^2} dA(w), \ \ 
z=x+iy\in{\mathbb R}_+^2, 
\] 
where 
$d A(w)$ denotes the Lebesgue measure and
${\mathbb R}_+^2$ is the upper half-space over 
$\mathbb{R}^2_+\simeq {\mathbb C}$. 
Karapetyants and Samko \cite{KS2017} proved that $B_{{\mathbb R}_+^2}$ is a projection from $L^{p(\cdot)} ({\mathbb R}_+^2)$ 
onto ${\mathcal A}^{p(\cdot)}({\mathbb R}_+^2)$ 
if $p(\cdot)\in{\mathcal P}({\mathbb R}_+^2)$ satisfies 
the log-H\"{o}lder condition and the log-decay condition (\cite[Theorem 3.1 (1)]{KS2017}). 
So they have obtained the norm inequality 
\begin{equation} 
\| B_{{\mathbb R}_+^2} f\|_{L^{p(\cdot)}({\mathbb R}_+^2)} 
\le C 
\| f\|_{L^{p(\cdot)}({\mathbb R}_+^2)} 
\label{eq 20191104-1}
\end{equation}
for all $f\in L^{p(\cdot)}({\mathbb R}_+^2)$.

Finally,
we consider $b_{{\mathbb R}_+^n}$,
the harmonic projection in ${\mathbb R}^n_+$. 
Let ${\mathbb R}^n_+$ stand for the upper half-space
over ${\mathbb R}^n$ with $n \ge 2$.
For
$x=(x_1,x_2,\ldots,x_n)$,
we write
$x'=(x_1,x_2,\ldots,x_{n-1})$
and
$\bar{x}=(x',-x_n)$.
As usual, $h^p({\mathbb R}^n_+)$ stands
for the harmonic Bergman space of harmonic functions
that belong to $L^p({\mathbb R}^n_+)$.
Here and below $dA(x)$ denotes the Lebesgue measure.
The corresponding Bergman projection
$b_{{\mathbb R}^n_+}$ defined by
%for $f \in L^p({\mathbb R}^n)$,
\begin{align*}
b_{{\mathbb R}^n_+}f(x)
&=
\int_{{\mathbb R}^n_+}R(x,y)f(y)dA(y)  \\
&=
\frac{2}{\pi^{\frac{n}{2}}}
\Gamma\left(\frac{n}{2}\right)
\int_{{\mathbb R}^n_+}
\frac{n(x_n+y_n)-|x-\bar{y}|^2}{|x-\bar{y}|^{n+2}}f(y)dA(y) ,
\end{align*}
is bounded from $L^p({\mathbb R}^n_+)$ onto
$h^p({\mathbb R}^n_+)$
(\cite{Zhu-book}).
Namely $b_{{\mathbb R}_+^n}f \in h^p({\mathbb R}_+^n)$
and the norm inequality 
\begin{equation} 
\| b_{{\mathbb R}_+^n}f \|_{L^p({\mathbb R}_+^n)} 
\le C \| f\|_{L^p({\mathbb R}^n)}
\label{eq 20191104-2}
\end{equation}
hold for all $f\in L^{p} ({\mathbb R}_+^n)$. 
Karapetyants and Samko have extended  \eqref{eq 20191104-2} in the variable exponent settings 
(\cite[Theorem 5.1 ]{KS2017}).

In the present paper, we consider the modular inequalities 
corresponding to the norm inequalities 
\eqref{intro-norm-analytic}, \eqref{eq 20191104-1} and \eqref{eq 20191104-2}. 
More precisely,
for example, if 
$p(\cdot)$ satisfies 
\[
1 < \mathop{\mathrm{ess \, sup}}_{z\in \mathbb{D}}p(z)
\le \mathop{\mathrm{ess \, sup}}_{z\in \mathbb{D}}p(z)<\infty
\]
 and the modular inequality
\[
\int_{\mathbb{D}} |Pf(z)|^{p(z)}dA(z)
\le C \int_{\mathbb{D}} |f(z)|^{p(z)}dA(z)
\]
holds for all $f\in L^{p(\cdot)}(\mathbb{D})$, 
then the variable exponent $p(\cdot)$ must be a constant function. 
We can prove similar results for $B_{{\mathbb R}_+^n}$ and $b_{{\mathbb R}_+^n}$. 
In order to prove them 
we need a lower bound for the 
image of the characteristic function
of a certain set.
We will show a key lemma
for the lower bound before the statement of the main results.

In the present paper we will use the following notation.
\begin{enumerate}
\item
Given a measurable set $E$, we denote the Lebesgue measure of $E$ by $|E|$.
We define the characteristic function of $E$ by $\chi_E$. 

\item
A symbol $C$ always stands for a positive constant independent of the main parameters. 
\end{enumerate}

%%%%%%%%%%%%%%%%%%%%%%%%%%%%%%%%%%%%%%%%%%%%%%%%%%%%%%%%%%%%%%%%%%
\section{Function spaces with variable exponent}

Let $\mathbb{D}$ be the open unit disk in the complex plane $\mathbb{C}$, that is, 
$$\mathbb{D}:= \{ z\in \mathbb{C} \, : \, |z|<1 \} .$$
Let also ${\mathbb R}^n_+$ be the upper half plane, that is,
\[
\mathbb{R}^n_+:=\{x=(x',x_n) \in {\mathbb R}^n\,:\,x' \in {\mathbb R}^{n-1}, x_n>0\}.
\]
In the present paper we concentrate on the 
theory on function spaces defined on $\mathbb{D}$
or ${\mathbb R}^n_+$ with $n \ge 2$.
We first define some fundamental notation on variable exponents.
Let $X$ denote either ${\mathbb D}$ or ${\mathbb R}_+^n$.

\begin{definition}
\
\begin{enumerate}
\item Given a measurable function $p(\cdot) : X \to [1,\infty)$, 
we define
\[
p_+:= \mathop{\mathrm{ess \, sup}}_{z\in X}p(z), 
\quad
p_-:= \mathop{\mathrm{ess \, inf}}_{z\in X}p(z).
\]
\item
The set $\mathcal{P}(X)$ consists of all measurable functions
$p(\cdot) : X \to [1,\infty)$ satisfying $1<p_-$ and $p_+<\infty$.
\end{enumerate}
\end{definition}

Chac\'on and Rafeiro \cite{ChaconRafeiro2014}
defined generalized 
Lebesgue spaces and
Bergman spaces on $\mathbb{D}$
with variable exponent.

\begin{definition}
Let $dA(z)$ 
be the normalized Lebesgue measure on $X$ and $p(\cdot)\in \mathcal{P}(X)$.
The Lebesgue space $L^{p(\cdot)}(X)$ consists of all measurable functions $f$ on $X$ 
satisfying that the modular 
\[
\rho_p(f):= \int_{X} |f(z)|^{p(z)}\, dA(z)
\]
is finite. 
The Bergman space $A^{p(\cdot)}(\mathbb{D})$ is the set of all
holomorphic functions $f$ on $\mathbb{D}$ such that $f\in L^{p(\cdot)}(\mathbb{D})$. 
\end{definition}

We note that $L^{p(\cdot)}(X)$ is a Banach space equipped with the norm
\[
\| f \|_{L^{p(\cdot)}(X)}:=
\inf \left\{ \lambda >0 \, : \, \rho_p(f/\lambda) \le 1 \right\} .
\]

The projection $P:L^2(\mathbb{D}) \to A^2(\mathbb{D})$ is called the Bergman projection and given by
\[
Pf(z)= \int_{\mathbb{D}} \frac{f(w)}{(1-\bar{w}z)^2} dA(w) .
\]
It is known that $P : L^p(\mathbb{D}) \to A^p(\mathbb{D})$
is bounded in the case that $1<p<\infty$ is a constant exponent
(\cite{HKZ-book,Zhu-book}).
See also \cite{SSbook} for the case of $p=2$.

Chac\'on and Rafeiro \cite[Theorem 4.4]{ChaconRafeiro2014} 
proved the following boundedness:
\begin{theorem}
Suppose that $p(\cdot) \in \mathcal{P}(\mathbb{D})$ satisfies the local
log-H\"older continuous condition
\[
|p(z_1)-p(z_2)| \le \frac{C}{\log (\mathrm{e}+1/|z_1-z_2|)}
\quad (z_1, \, z_2\in \mathbb{D}) .
\] 
Then the Bergman projection
$P$ is bounded from $L^{p(\cdot)}(\mathbb{D})$ to $A^{p(\cdot)}(\mathbb{D})$, 
in particular, the norm inequality
\[
\| Pf \|_{L^{p(\cdot)}(\mathbb{D})} \le C \, \| f \|_{L^{p(\cdot)}(\mathbb{D})} 
\] 
holds for all $f\in L^{p(\cdot)}(\mathbb{D})$.
\end{theorem}

In the following sections, we consider the modular inequalities 
corresponding to the norm inequalities 
\eqref{intro-norm-analytic}, \eqref{eq 20191104-1} and \eqref{eq 20191104-2}.

\section{Bergman projection on ${\mathbb D}$}

\begin{theorem}
\label{Theorem-main1}
Let $p(\cdot) \in \mathcal{P}(\mathbb{D})$.
If the modular inequality 
\begin{equation}
\int_{\mathbb{D}} |Pf(z)|^{p(z)}dA(z)
\le C \int_{\mathbb{D}} |f(z)|^{p(z)}dA(z)
\label{eq 180908}
\end{equation}
holds for all $f\in L^{p(\cdot)}(\mathbb{D})$, 
then $p(z)$ equals to a constant for almost every $z\in \mathbb{D}$.
\end{theorem}

In order to prove this theorem, 
we apply the following lower pointwise estimate for the Bergman projection.

\begin{lemma}
\label{Lemma-bergman-estimate}
Let $\tau \in {\mathbb D}$.
Then there exists a compact neighborhood $K_{\tau}$ of $\tau$ such that
\[
{\rm Re}(P\chi_{E}(z)) \ge c_\tau|E|
\]
for all measurable sets $E \subset K_\tau$,
where $c_\tau$ is a positive constant depending only on $\tau$.
\end{lemma}
\begin{proof}
Note that there exists a compact neighborhood $K_\tau$
of $\tau$ such that
\[
c_\tau:=\inf_{z,w \in K_\tau}{\rm Re}\left(\frac{1}{(1-\bar{w}z)^2}\right)>0.
\]
Thus,
\[
{\rm Re}(P\chi_E(z))= \int_{E} {\rm Re}\left(\frac{1}{(1-\bar{w}z)^2}\right) dA(w)
\ge c_\tau\int_{E}dA(w)=c_\tau|E|,
\]
as required.
\end{proof}

Now we prove Theorem \ref{Theorem-main1}. 

\begin{proof}[Proof of Theorem \ref{Theorem-main1}]
Let $\tau \in {\mathbb D}$
and $K_{\tau}$ be the compact neighborhood appearing in Lemma \ref{Lemma-bergman-estimate}.
Assume that 
$p(z)$ does not equal to any constant for almost every $z\in K_\tau$. 
Then we can find subsets $K_\tau^\pm$ of $K_\tau$ such that
\begin{equation}\label{eq:180826-1}
\sup_{z \in K_\tau^-}p(z)<\inf_{z \in K_\tau^+}p(z).
\end{equation}
Using Lemma \ref{Lemma-bergman-estimate} and modular inequality (\ref{eq 180908}), 
we have 
\[
\int_{K_\tau^+}(k c_\tau|K_\tau^-|)^{p(z)}\,dA(z)
\le
\int_{K_\tau^+}|k P\chi_{K_\tau^-}(z)|^{p(z)}\,dA(z)
\le C
\int_{{\mathbb D}}(k \chi_{K_\tau^-})^{p(z)}\,dA(z)
\]
for all $k>0$.
Consequently, if $k c_\tau|K_\tau^-|>1$ and $k>1$,
then we obtain 
\[
|K_\tau^+|(k c_\tau|K_\tau^-|)^{
 \mathop{\mathrm{ess \, inf}}_{z \in K_\tau^+}p(z)}
\le C
|K_\tau^-| k^{
 \mathop{\mathrm{ess \, sup}}_{z \in K_\tau^-}p(z)}.
\]
This contradicts to
(\ref{eq:180826-1}).
Consequently, it follows that
for all $\tau \in {\mathbb D}$ there exists a compact neighborhood
$K_\tau$ such that $p(z)$ equals to a constant for almost every $z\in K_\tau$.
Since ${\mathbb D}$ is connected, it follows that
$p(z)$ equals to a constant for almost every $z\in {\mathbb D}$.
\end{proof}

%%%%%%%%%%%%%%%%%%%%%%%%%%%%%%%%%%%%%%

%%%%%%%%%%%%%%%%%%%%%%%%%%%%%%%%%%%%%%%%%%%%%%%%%%%

\section{Bergman projection onto ${\mathbb R}^2_+$}

\if0
We identify ${\mathbb R}^2_+$ with $\{z \in {\mathbb C}\,:\,{\rm Re}(z)>0\}$.
Karapetyants and Samko \cite{KS2017} considered the following operator;
\[
B_{{\mathbb R}^2_+}f(z) = \frac{-1}{\pi} \int_{{\mathbb R}_+^2}\frac{f(w)}{(z-\overline{w})^2} dA(w), \ \ z\in{\mathbb R}_+^2. 
\] 
They proved that $B_{{\mathbb R}_+^2}$ is a projection from $L^{p(\cdot)} ({\mathbb R}_+^2)$ 
onto ${\mathcal A}^{p(\cdot)}({\mathbb R}_+^2)$ 
if $p(\cdot)\in{\mathcal P}({\mathbb R}_+^2)$ satisfies 
the log-condition and decay condition (\cite[Theorem 3.1 (1)]{KS2017}). 
\fi

As the following lemma shows,
$B_{{\mathbb R}^2_+}$ is not degenerate.
\begin{lemma} 
\label{lemma 20191031}
Let $\tau\in{\mathbb R}_+^2$. Then there exists a compact neighborhood $K_{\tau}$ of $\tau$ such that 
\[
{\rm Re} \left( B_{{\mathbb R}_+^2}(\chi_E) (z) \right) \ge C_{\tau} |E| 
\]
for all measurable sets $E \subset K_{\tau}$. 
\end{lemma}

\begin{proof} 

Let $\tau = \alpha+\beta i \in \mathbb{C}\simeq {\mathbb R}_+^2$. 
Firstly, we prove that there exist a compact neighborhood $K_{\tau}$ of $\tau$ such that 
\[
{\rm Re}\left( \frac{1}{(z-\overline{w})^2} \right) <0
\]
holds for any $z, w\in K_{\tau}$. 
To do this, we consider the real part of $(\overline{z}-w)^2$ 
keeping in mind that
\[
{\rm Re}\left( \frac{1}{(z-\overline{w})^2} \right) 
= 
{\rm Re}\left( \frac{ (\overline{z}-w)^2 }{|z-\overline{w}|^4} \right). 
\]
We can take $\gamma >0$
so that 
$\beta-\gamma>0$
because $\beta >0$.
We learn that
\[
K_{\tau} = \{ x+yi\,:\, \alpha-(\beta-\gamma)/2 \le x \le \alpha+(\beta-\gamma)/2, \, \beta-\gamma\le y \le \beta+\gamma\} (\subset {\mathbb R}_+^2)
\]
does the job.
Let $z=a+bi$, $w=c+di \in K_{\tau}$. 
It is easy to see that ${\rm Re} (\overline{z}-w)^2 <0$ 
since
\[
(\overline{z}-w)^2 = (a-c)^2-(b+d)^2-2(a-c)(b+d)i
\]
and $|a-c|\le \beta-\gamma < 2(\beta-\gamma)\le |b+d|$. 

Finally, we have 
\[
{\rm Re}(B_{{\mathbb R}_+^2} (\chi_E(z))) 
= 
\frac{-1}{\pi} \int_E {\rm Re} \left( \frac{1}{(z-\overline{w})^2} \right) dA(w) 
\ge C_{\gamma} \int_E \,dA(w) =c_{\gamma} |E|
\]
for any $E \subset K_{\tau}$. 
\end{proof}

Using 
Lemma \ref{lemma 20191031} and 
an argument similar 
to the proof of Theorem \ref{Theorem-main1}, 
we obtain the following theorem. So we omit the proof. 
\begin{theorem} 
Let $p(\cdot)\in{\mathcal P}({\mathbb R}_+^2)$. 
If the modular inequality 
\[
\int_{ {\mathbb R}_+^2} \left| B_{{\mathbb R_+^2}}f(z) \right|^{p(z)} \,d A(z) 
\le 
C \int_{{\mathbb R}_+^2} |f(z)|^{p(z)}\,dA(z)
\] 
holds for all $f\in L^{p(\cdot)}({\mathbb R}_+^2)$, 
then $p(z)$ equals to a constant for almost every $z\in \mathbb{R}_+^2$. 
\end{theorem}

\section{Harmonic projection in ${\mathbb R}^n_+$}

The same technique can be applied to the harmonic projection over ${\mathbb R}_+^n$. 

\begin{theorem} 
Let $p(\cdot)\in{\mathcal P}({\mathbb R}_+^n)$. 
If the modular inequality 
\[
\int_{ {\mathbb R}_+^n} \left| b_{{\mathbb R_+^n}}f(z) \right|^{p(z)} \,d A(z) 
\le 
C \int_{{\mathbb R}_+^n} |f(z)|^{p(z)}\,dA(z)
\] 
holds for all $f\in L^{p(\cdot)}({\mathbb R}_+^n)$, 
then  $p(z)$ equals to a constant for almost every $z\in \mathbb{R}_+^n$. 
\end{theorem}

\begin{proof}
Let $x=(x',x_n) \in {\mathbb R}^n_+$ be fixed.
Then we have
\[
\frac{n(x_n+z_n)-|x-\bar{z}|^2}{|x-\bar{z}|^{n+2}}
=
\frac{n-1}{2^{n+2}}x_n{}^{-n}
\]
for $z=(z',z_n)=x=(x',x_n)$. Thus 
we obtain
\[
\frac{n(x_n+y_n)-|x-\bar{y}|^2}{|x-\bar{y}|^{n+2}}
>
\frac{n-1}{2^{n+3}}x_n{}^{-n}
\]
as long as $y=(y',y_n)$ belongs to an open neighborhood $U$
of $x$.
Thus, if we go through the same argument as before,
we see that 
$p(z)$ equals to a constant for almost every $z\in \mathbb{R}_+^n$. 
\end{proof}

\section*{Acknowledgements} 
\noindent
Mitsuo Izuki was partially supported by Grand-in-Aid for Scientific Research (C), No.\,15K04928, for Japan Society for the Promotion of Science. 
Takahiro Noi was partially supported by Grand-in-Aid for Young Scientists (B), No.\,17K14207, for Japan Society for the Promotion of Science. 
Yoshihiro Sawano was partially supported by Grand-in-Aid for Scientific Research (C), No.\,19K03546, for Japan Society for the Promotion of Science. 
This work was partly supported by Osaka City University Advanced Mathematical Institute (MEXT Joint Usage/Research Center on Mathematics and Theoretical Physics). 
The authors are thankful to
Professor Alexey Karapetyants 
for letting us know about some literatures
in the theory of Bergman spaces.

%%%%%%%%%%%%%%%%%%%%%%%%%%%%%%%%%%%%%%%%%%%%%%%%%%%%%%%%%%%%%%%%%%%%%%%

Mitsuo Izuki (Corresponding author),\\
Faculty of Liberal Arts and Sciences, \\
Tokyo City University, \\
1-28-1, Tamadutsumi Setagaya-ku Tokyo 158-8557, Japan. \\ 
E-mail: izuki@tcu.ac.jp

\smallskip

Takahiro Noi,\\
Department of Mathematics and Information Science,\\
Tokyo Metropolitan University,\\
Hachioji, 192-0397, Japan.\\
E-mail: taka.noi.hiro@gmail.com

\smallskip

Yoshihiro Sawano,\\
Department of Mathematics and Information Science,\\
Tokyo Metropolitan University,\\
Hachioji, 192-0397, Japan.\\
E-mail: yoshihiro-sawano@celery.ocn.ne.jp
\end{document}